\documentclass{article}

\usepackage[russian,english]{babel}

\usepackage[tbtags]{amsmath}
\usepackage{amsfonts,amssymb,mathrsfs,amsmath,amscd,comment}

\usepackage[russian,english]{babel}
\usepackage[paper]{interper2e}

\VOL{31} \NO{1} \YEAR{2024} \PAGES{107--112}
\let\rom=\textup

\usepackage{amsfonts,amsmath, amssymb}

\usepackage{xcolor}

\usepackage{enumerate}
\usepackage{epsf,epsfig,amsfonts}
\usepackage{amsfonts}
\usepackage{amsmath,amssymb,amsthm}
\usepackage{amsmath,amssymb,amsthm}

\usepackage{amssymb}
\usepackage{amsthm}
\usepackage{epsf,epsfig,amsfonts,graphicx}

\newtheorem{Theorem}{Theorem}[section]

\newtheorem{Lemma}{Lemma}[section]
\newtheorem{Proposition}{Proposition}[section]

\theoremstyle{remark}

\newcommand{\be}{\begin{equation}}
\newcommand{\ee}{\end{equation}}

\newcommand{\ddd}{\mathrm{d}}

\newcommand{\pd}[2]{\frac{\partial#1}{\partial#2}}

\newcommand{\weg}[1]{}

\begin{document}

\title{On the Linearization of Certain Singularities\\of Nijenhuis Operators}
\author{A. Yu.~Konyaev$^{*,1}$}
\affil{$^*$Faculty of Mechanics and Mathematics, Moscow State University, Moscow,
Russia\\
$^{**}$Moscow Center for Fundamental and Applied Mathematics, 119992, Moscow, Russia\\
$^1$E-mail: {\tt maodzund@ya.ru}
}

\date{Received December 20, 2023; Revised December 20, 2023; Accepted December 25, 2023}

\maketitle

\begin{abstract}
We consider a linearization problem for Nijenhuis operators in dimension two
around a point of scalar type in analytic category. The problem was almost
completely solved in \cite{nij2}. One case, however, namely the case of
left-symmetric algebra $\mathfrak b_{1, \alpha}$, proved to be difficult. Here
we solve it and, thus, complete the solution of the linearization problem for
Nijenhuis operators in dimension two. The problem turns out to be related to
classical results on the linearization of vector fields and their monodromy
mappings.

\medskip
\noindent {\bf DOI} 10.1134/S106192084010084

\end{abstract}

\begin{fulltext}

\section{Introduction}

Let $\mathfrak a$ be an algebra (finite- or infinite-dimensional) with an
operation $\star$ over the field of $\mathbb R$ or $\mathbb C$. The associator
is a trilinear operation, defined by $\mathcal A(\xi, \eta, \zeta) = (\xi \star
\eta) \star \zeta - \xi \star (\eta \star \zeta)$ for arbitrary triple $\xi,
\eta, \zeta \in \mathfrak a$. The associator identically vanish if and only if
the algebra $\mathfrak a$ is associative. If $\mathcal A(\xi, \eta, \zeta) =
\mathcal A(\eta, \xi, \eta)$, i.e., the associator is symmetric in the first
two indices, then the algebra $\mathfrak a$ is called a left-symmetric algebra.

Left-symmetric algebras were introduced independently by Vinberg \cite{vin} and
Kozul \cite{koz} in purely algebraic context. Later they have occurred in the
theory of infinite-dimensional integrable systems [3--5]. See \cite{burde2} for
overview on the variety of applications in algebra, geometry, and mathematical
physics.

The essential property of left-symmetric algebras is that the commutator $[\xi,
\eta] = \xi \star \eta - \eta \star \xi$ satisfies the Jacobi identity and,
thus, defines a Lie algebra structure. We call the corresponding algebra an
associated Lie algebra.

Any finite-dimensional algebra over $\mathbb R$ or $\mathbb C$ can be treated
as an affine space, equipped with operator field $R_\eta \xi = \xi \star \eta$,
where $\eta \in \mathfrak a$ and $\xi \in T_\eta \mathfrak a \cong \mathfrak
a$.

\begin{Theorem}[\cite{win, nij2}]\label{t0}
A finite-dimensional algebra $\mathfrak a$ over $\mathbb R$ or $\mathbb C$ is a
left-symmetric algebra if and only if the Nijenhuis torsion of $R$ \rom(we omit
the dependence on the point\/\rom)
$$
\mathcal N_R (\xi, \eta) = R[R\xi, \eta] + R[\xi, R \eta] - [R\xi, R \eta]
- R^2[\xi, \eta]
$$
vanishes for all vector fields $\xi, \eta$.
\end{Theorem}
The right-hand side of the above formula defines a tensor of type $(1, 2)$,
which was introduced in \cite{nijenhuis}. The operator field with vanishing
Nijenhuis torsion is called a Nijenhuis operator.

Theorem \ref{t0} establishes that the left-symmetric algebras play the same
role in Nijenhuis geometry as Lie algebras play in Poisson geometry: recall
that Lie algebras are in one-to-one correspondence with Poisson tensors on an
affine space that are homogeneous and linear in coordinates.

In \cite{wein}, Weinstein introduces the linearization problem for Poisson
brackets, vanishing at a point. Following the similar line of thought, one
might introduce the linearization problem for Nijenhuis operators.

Consider a manifold $\mathrm M^n$ and a Nijenhuis operator $R$ with the
property that, at a given point $\mathrm p$, it is proportional to a scalar,
i.e., $\lambda_0 \operatorname{Id}$ for some constant $\lambda_0$. We call such
points singular points of scalar type. The Taylor series for $R$ around
$\mathrm p$ is
$$
R = \lambda_0 \operatorname{Id} + R_1 + \dots
$$
In \cite{nij2} it is shown that $R_1$ is Nijenhuis and defines on $T_{\mathrm
p} \mathrm M^n$ a natural (i.e., independent of the coordinates!) structure of
left-symmetric algebra. In given coordinates $x^1, \dots, x^n$, the structure
constants of this algebra are given by $a_{ij}^k = \pd{R^k_i}{x^j} (\mathrm
p)$.

We say that $R$ is linearizable if there exists a coordinate change such that
$R = \lambda_0 \operatorname{Id} + R_1$. We say that a left-symmetric algebra
$\mathfrak a$ is nondegenerate if, for any Nijenhuis operator $R$ with linear
part $R_1$ given by $\mathfrak a$, the linearization around $\mathrm p$ exists.
The algebra is called degenerate otherwise. The question can be posed for
formal, analytic, and smooth operator fields and coordinate changes.

Notice that the existence of nondegenerate left-symmetric algebras is not
obvious even in the formal category: in the Weinstein case, the formal
linearization heavily relied on the existence of a Lie group. We do have a Lie
group (the one corresponding to the associated Lie algebra), but it is little
or does not help. For example, as we see below, the linearization problem can
be formulated and, in certain cases, solved for commutative associated Lie
algebras (see \cite{nij1} for a special case in arbitrary dimension).

In \cite{nij2}, this problem of finding nondegenerate left-symmetric algebras
was approached in dimension two. First, the classification of left-symmetric
algebras in dimension two was obtained. Here we give only the corresponding
Nijenhuis operators $R$, which linearly depend on the coordinates $x, y$:
\begin{equation}\label{twodim}
\begin{aligned}
& \mathfrak b_{1, \alpha}:
\left(\begin{array}{cc} 0 & x \\ 0 & \alpha y\\ \end{array}\right),
\quad \begin{array}{l}
     \mathfrak b_{2, \beta}\\
     \beta \neq 0
\end{array}:
\left(\begin{array}{cc} y & \big(1 - \frac{1}{\beta}\big) x \\
0 & y\\ \end{array}\right), \quad \mathfrak b_3:
\left(\begin{array}{cc} 0 & x + y \\ 0 & y\\ \end{array}\right), \\
& \mathfrak b_{4}^+:
\left(\begin{array}{cc} 0 & -x \\ x & - 2y\\ \end{array}\right),
\quad \mathfrak b_4^-:
\left(\begin{array}{cc} 0 & - x \\ - x & -2y\\ \end{array}\right),
\quad \mathfrak b_5: \left(\begin{array}{cc} y & y \\ 0 & y\\ \end{array}\right),
\quad \mathfrak c_0: \left(\begin{array}{cc} 0 & 0 \\ 0 & 0\\ \end{array}\right), \\
& \mathfrak c_2: \left(\begin{array}{cc} 0 & 0 \\ 0 & y\\ \end{array}\right),
\quad \mathfrak c_3: \left(\begin{array}{cc} 0 & y \\ 0 & 0\\ \end{array}\right),
\quad \mathfrak c_4: \left(\begin{array}{cc} y & x \\ 0 & y\\ \end{array}\right),
\quad \mathfrak c_5^+: \left(\begin{array}{cc} y & x \\ x & y\\ \end{array}\right),
\quad \mathfrak c_5^-: \left(\begin{array}{cc} y & x \\ - x & y\\ \end{array}\right).
\end{aligned}
\end{equation}
The letter $\mathfrak c$ stands for the commutative associated Lie algebra and
$\mathfrak b$ for the noncommutative one (up to an isomorphism, there are
exactly two Lie algebras).

Second, in the smooth category, the algebras from the list \eqref{twodim} were
classified in terms of degeneracy/nondegeneracy. Define $\Sigma_{\mathrm {sm}}$
as a union of three sets: $\{\alpha \vert \alpha \leq 0\}$, $\{r \vert r \in
\mathbb N, r \geq 3\}$ and $\{1/m \vert m \in \mathbb N, r \geq 2\}$. The
result (Theorem~1.3 in~\cite{nij2}) is as follows.
$$
\begin{aligned}
&\text{Degenerate}: \mathfrak c_1, \mathfrak c_2, \mathfrak c_3,
\mathfrak c_4, \mathfrak b_5, \mathfrak b_{2, \beta}
\mathfrak b_{1, \alpha} \, \text{\rm{for}} \,  \alpha \in \Sigma_{\mathrm{sm}}, \\
& \text{Nondegenerate}: \mathfrak b^+_4, \mathfrak b^-_4, \mathfrak c^+_5,
\mathfrak c^-_5, \mathfrak b_3, \mathfrak b_{1, \alpha} \, \text{\rm{for}} \,
\alpha \notin \Sigma_{\mathrm{sm}}.
\end{aligned}
$$
We see that, in the smooth case, the classification was complete. At the same
time, in the analytic case, only a partial classification was obtained. Let
$[q_0, q_1, q_2,\dots]$ be a decomposition of an irrational number $\alpha$
into the continuous fraction. If the series
\begin{equation*}
B(x) = \sum \limits_{i = 0}^{\infty}\frac{\operatorname{log}q_{i + 1}}{q_i}
\end{equation*}
converges, then $\alpha$ is a \textbf{Brjuno number}. We denote the set of
negative irrational numbers which are not Brjuno numbers by $\Sigma_u$. The
Lebesgue measure of $\Sigma_{\mathrm{u}}$ is zero. In addition, denote  the
union of negative rational numbers, zero, $\{r \vert r \in \mathbb N, r \geq
3\}$, and $\{1/m \vert m \in \mathbb N, r \geq 2\}$ by $\Sigma_{\mathrm{an}}$.
Notice, that $\Sigma_{\mathrm{an}} \subset \Sigma_{\mathrm{sm}}$. Then Theorem
1.4 of~\cite{nij2} yields:
$$
\begin{aligned}
&\text{Degenerate}: \mathfrak c_1, \mathfrak c_2, \mathfrak c_3,
\mathfrak c_4, \mathfrak b_5, \mathfrak b_{2, \beta}
\mathfrak b_{1, \alpha} \, \text{\rm{for}} \,  \alpha \in \Sigma_{\mathrm{an}}, \\
& \text{Nondegenerate}: \mathfrak b^+_4, \mathfrak b^-_4, \mathfrak c^+_5,
\mathfrak c^-_5, \mathfrak b_3, \mathfrak b_{1, \alpha} \, \text{\rm{for}} \,
\alpha \notin \Sigma_{\mathrm{an}} \cup \Sigma_{\mathrm u}, \\
& \text{Unknown}: \mathfrak b_{1, \alpha} \, \text{\rm{for}} \,
\alpha \in \Sigma_{\mathrm{u}}.
\end{aligned}
$$
The main purpose of this paper is to close the aforementioned gap. The main
result is as follows.

\begin{Theorem}\label{t1}
The left-symmetric algebra $\mathfrak b_{1, \alpha}$ is degenerate for $\alpha
\in \Sigma_{\mathrm u}$ in the analytic category.
\end{Theorem}

This completes the classification problem of left-symmetric algebras in
dimension two for analytic and smooth category. The author is grateful to
Yu.~Ilyashenko for the important ideas for the proof and attention to this work. The work is supported by grant Russian Science Foundation, RScF 24-21-00450.


\section{Proof of Theorem \ref{t1}}

Recall that the Fr\"ohlicher--Nijenhuis bracket of the operator fields is given
by the formula
\begin{equation}\label{fn}
\begin{aligned}[]
 [[R, Q]] (\xi, \eta) = Q[R\xi, \eta] + & R[\xi, Q\eta]
 - [R\xi, Q\eta] - RQ[\xi, \eta] \\
+ & R[Q\xi, \eta] + Q[\xi, R\eta] - [Q\xi, R\eta] - QR[\xi, \eta]
\end{aligned}
\end{equation}
This is a tensor of type $(1, 2)$ which is skew-symmetric in the lower indices.
There are several obvious properties that are crucial for our proof.
\begin{enumerate}
    \item If the entries of $R$ are homogeneous polynomials of degree $k$ and
    the entries of $Q$ are homogeneous polynomials of degree $l$, then the
    entries of $[[R, Q]]$ are homogeneous polynomials of degree $k + l - 1$.
    This follows immediately from the formula
    \item $[[R, Q]] = [[Q, R]]$ and $[[R, R]] = 2 \mathcal N_R$.
\end{enumerate}
Now, let us proceed with a proof.

Consider a two-dimensional manifold $\mathrm M^2$ equipped with Nijenhuis
operator $R$. Let $\mathrm p \in \mathrm M^2$ be such that $R = \lambda_0
\operatorname{Id}$ at this point. The statement we are proving is local in
nature, so we may assume that $\mathrm M^2 = D^2$, i.e., the two-dimensional
disk. The Taylor composition at this point is $R = \lambda_0 \operatorname{Id}
+ R_1 + \dots$. Here $R_1$ is the linear part associated with $\mathfrak b_{1,
\alpha}$ for $\alpha \in \Sigma_{\mathrm u}$.

\begin{Proposition}\label{p1}
There exists formal coordinates $x, y$ around $\mathrm p$ such that, in these
coordinates,
\begin{equation}\label{prep}
R = \left(\begin{array}{cc}
     \lambda_0 & 0 \\
     0 & \lambda_0
\end{array}\right) + \left(\begin{array}{cc}
     0 & p(x, y) \\
     0 & q(y)
\end{array}\right)
\end{equation}
with $p(x, y) = x + \{\text{terms of order $>$ 1}\}$ and $q(y) = \alpha y +
\{\text{terms of order $>$ 1}\}$.
\end{Proposition}
\begin{proof}
First, we can replace $R$ by $R - \lambda_0 \operatorname{Id}$. The resulting
operator is Nijenhuis, and it is enough to prove the statement of the
proposition for such operators only. Take coordinates $u, v$ around $\mathrm p$
such that the linear part $R_1$ of the Taylor decomposition is in the form
given in Table \eqref{twodim}. We introduce
$$
L_i = \left(\begin{array}{cc}
     0 & l_i(u, v)  \\
     0 & m_i(u)
\end{array}\right), \quad R_k = \left(\begin{array}{cc}
     a(u, v) & b(u, v)  \\
     c(u, v) & d(u, v)
\end{array}\right),
$$
where $l_i, m_i $ are homogeneous polynomials of degree $i$ and $a, b, c, d$
are homogeneous polynomials of degree $k$. Assume that the Taylor decomposition
of $R$ at $\mathrm p$ is in the form $R = R_1 + L_2 + \dots + L_{k - 1} + R_k +
\dots$, where $R_k$ is in the aforementioned form. Note that $L_i
\partial_{u} = R_1 \partial_u = 0$. Due to formula \eqref{fn}, we obtain
$$
[[R_i, L_i]](\partial_u, \partial_v) = [[L_i, L_j]] (\partial_u, \partial_v) = 0
$$
In dimension two, this implies that both brackets vanish identically. The first
property of the Fr\"ohlicher--Nijenhuis bracket takes the form
$$
0 = [[R, R]] = [[R_k, R_1]] + \text{terms of the order $> k$}
$$
By the direct computation,
$$
\begin{aligned}
 [[R_k, R_1]](\partial_u, \partial_v)
 =& R_1[\partial_u, R_k \partial_v]  + R_1[R_k \partial_u, \partial_v]
 + R_k [\partial_u, R_1\partial_v] - [R_k \partial_u, R_1 \partial_v] \\
 = & R_1[\partial_u, b \partial_u + d \partial_v] + R_1[a \partial_u
 + c \partial_v, \partial_v] + R_k [\partial_u, u\partial_u
 + \alpha v \partial_v] - [a\partial_u + c \partial_v, u \partial_u
 + \alpha v \partial_v] \\
 = & u d_u \partial_u + \alpha v d_u \partial_v - u c_v \partial_u
 - \alpha v c_v \partial_v  + a \partial_u + c \partial_v - (a - u a_u
 - \alpha v a_v) \partial_u  \\
 &-  (b - u b_u - \alpha v b_v) \partial_v = 0.
\end{aligned}
$$
Gathering the similar terms, we obtain the system of two equations
\begin{equation}\label{base}
\begin{aligned}
  u d_u - u c_v + u a_u + \alpha v a_v &= 0, \\
  \alpha v d_u + u c_u + (1 - \alpha) c &= 0.
\end{aligned}
\end{equation}
The next is a standard approach to a formal linearization. Consider the
coordinate change in the form $\bar u = u + f(u, v), \bar u = v = g(u, v)$,
where $f, g$ are homogeneous polynomials of degree $k$. The inverse coordinate
change $u(\bar u, \bar v), v(\bar u, \bar v)$ is a formal series in the form $u
= \bar u + \{\text{terms of order $\geq k$}\}$ and $v = \bar v + \{\text{terms
of order $\geq k$}\}$.

The next step is to recalculate the operator field in the new coordinates. We
first recalculate
$$
(\operatorname{Id} + J)^{-1} R (\operatorname{Id} + J)
= R_1 + L_2 + \dots + (R_k - J R_1 + R_1 J) + \dots.
$$
Here
$$
J = \left( \begin{array}{cc}
     f_u & f_v  \\
     g_u & g_v  \\
\end{array}\right).
$$
And then one needs to substitute $u (\bar u, \bar v), v (\bar u, \bar v)$ into
the formula. We omit the bar and see that $R_k$ is changed as
$$
R_k = \left(\begin{array}{cc}
     a - u g_u & \star  \\
     c - \alpha v g_u & \star
\end{array}\right)
$$
The terms we denoted as $\star$ are transformed in a more complicated manner.
Let $a = \sum_{i = 0}^k a_i u^i v^{k - i}$, where $a_i$ are constants. Taking
$g = \sum_{i = 1}^{k} \frac{1}{i} a_i u^i v^{k - i}$, we see that $a - u g_u =
a_0 v^k$. Thus, we may assume that $a$ depends only on $v$. The first equation
of \eqref{base} takes the form
$$
u(d_u - c_v) + \alpha k a_0 v^{k} = 0.
$$
We know that $(d_u - c_v)$ is a polynomial of degree $k - 1$ and $u(d_u - c_v)$
does not contain the term $v^k$. This implies that the polynomial identically
vanish if $a_0 = 0$. Thus, $a = 0$ and $d_u = c_v$. Substituting the last
equation into  the second identity of \eqref{base} for $c = \sum_{i = 0}^k c_i
u^i v^{k - i}$, we obtain
$$
u c_u + \alpha v c_v + (1 - \alpha c)
= \sum_{i = 0}^k c_i \big(i + 1 + \alpha (k - i - 1)\big) u^i v^{k - i} = 0.
$$
Since $\alpha$ is irrational and $i + 1 > 0$, we see that the equation in the
brackets cannot be zero. Thus, $c_i = 0$ and $c = 0$ as well. This implies that
$d_u = 0$ and we see that $R_k = L_k$.

Repeating this process and taking the composition of the countable number of
coordinate changes, we obtain a formal coordinate change, which transforms $R$
into the form \eqref{prep}. Proposition \ref{p1} is proved.
\end{proof}

Consider the equation
\begin{equation}\label{eq}
\operatorname{det}(R - \alpha \mu \operatorname{Id}) = 0.
\end{equation}
This is an analytic equation with respect to $\mu$. In Proposition \ref{p1},
the formal series $\frac{1}{\alpha}q(y)$, written as $\frac{1}{\alpha}q(y(u,
v))$ in initial coordinates, yields a formal series for an eigenvalue of $R$ in
the initial coordinates $u, v$. Thus, equation \eqref{eq} has a formal
solution.

Due to Artin's theorem \cite{artin}, there exists an analytical solution, for
which one can choose the first term to coincide with formal one. We see that
there exists an analytic function $\mu(u, v) = v + \dots$ for which $\alpha
\mu$ is an eigenvalue of $R$ and $\ddd \mu \neq 0$ at the coordinate origin.

\begin{Lemma}\label{lm1}
In any coordinates $\lambda, \mu$, where $\mu$ is a solution of \eqref{eq} with
$\ddd \mu \neq 0$, the operator field $R$ has the form
\begin{equation}\label{prep2}
R = \left(\begin{array}{cc}
     \lambda_0 & 0 \\
     0 & \lambda_0
\end{array}\right) + \left(\begin{array}{cc}
     0 & h(\lambda, \mu) \\
     0 & \alpha \mu
\end{array}\right),
\end{equation}
where $h(\lambda, \mu) = \lambda + \{\text{terms of higher order}\}$.
\end{Lemma}
\begin{proof}
First, let us show that, if $R$ is in the form \eqref{prep2} in one coordinate
system $\lambda, \mu$, then it has the same form in all such systems. Indeed,
consider $\bar \lambda = g(\lambda, \mu), \bar \mu = \mu$. The operator field
$R$ is transformed by the rule
\begin{equation}\label{chng}
\begin{aligned}
     \left(\begin{array}{cc}
        \pd{g}{\lambda} & \pd{g}{\mu}  \\
         0 & 1
    \end{array}\right)
    \left(\begin{array}{cc}
0 & h(\lambda, \mu) \\
0 & \alpha \mu \\
\end{array}\right)& \left(\begin{array}{cc}
        \frac{1}{\pd{g}{\lambda}} & - \frac{\pd{g}{\mu}}{\pd{g}{\lambda}}  \\
         0 & 1
    \end{array}\right) = \left(\begin{array}{cc}
0 & \pd{g}{\lambda} h(\lambda, \mu) + \alpha \pd{g}{\mu} \mu  \\
0 & \alpha \mu \\
\end{array}\right).
\end{aligned}
\end{equation}
Note that the coordinate change does not need to be analytic. That is, if we
have coordinates $\lambda, \mu$ in which the operator field $R$ is in the form
\ref{prep2}, then, taking the inverse change, we see that it was in the same
form in the initial coordinates.

Now, consider the formal coordinate change that exists due to Proposition
\ref{p1}. Consider the series $q(y)$. By construction, its Taylor decomposition
is in the form $q(y) = \alpha y + \dots$. Taking $\mu = \frac{1}{\alpha} q(y(u,
v))$ and $\lambda = x(u, v)$ and performing the same calculations as above, we
see that $R$ is in the form \eqref{prep2}. Thus, the lemma is proved.
\end{proof}

Lemma \ref{lm1} reduces the linearization problem to the linearization in
specific coordinates. Thus, we may assume that, in coordinates $u, v$, the
operator field $R$ has the form \eqref{prep2}. If a linearizing coordinate
change $x(u, v), y(u, v)$ exists, then we have $y = v =
\frac{1}{\alpha}\operatorname{tr} R$. Thus, the linearization transformation in
these coordinates is always triangular: $x = g(u, v)$ and $y = v$.

\begin{Proposition}\label{p2}
There exists an analytical function $h(u, v) = u + \dots$ such that the vector
field $\xi = (h(u, v), \alpha v)$ is not linearizable by the coordinate
transformations $x = g(u, v)$ and $y = v$ in a neighborhood of the critical
point $u = v = 0$.
\end{Proposition}

\begin{proof}
First, let us recall some basic facts and notions from the theory of ODEs (we
follow \cite{iy}). First, we complexify the system
\begin{equation}\label{sys}
\begin{aligned}
 & \dot u = h(u, v), \\
 & \dot v = \alpha v.
\end{aligned}
\end{equation}
If there exist finitely many solutions which can be continued into a critical
point, then they are called separatrices.

The second equation of \eqref{sys} can be solved, that is, $v(t) = c_1
e^{\alpha t}$. Substituting this into the first equation, we obtain a
one-dimensional nonstationary system $\dot u = h(u, t)$. Thus, in our case,
there are two separatrices, each of which is a complex punctured coordinate
line $u = 0$ and $v = 0$.

Given a separatrix $S$, one can define a mononodromy mapping. Choose a closed
curve $\gamma(t), t \in [0, 1]$ in $S$, which is $\mathbb C = \mathbb R$
without a point, which encompasses the coordinate origin in the positive
direction exactly once. For every point of $\gamma(t)$, choose a transversal
space $T_t$ to $S$, which at least continuously depend on the parameter $t$
with the property $T_0 = T_1$. In our case, $T_t = \mathbb R^2 = \mathbb C$ for
all $t$. Due to the construction, $\gamma(0) \in T_0$.

Due to the continuous dependence of the solution on the initial condition for
any sufficiently small neighborhood $U(\gamma(0)) \in T_0$, there exists a
curve $\bar \gamma (t)$ (not necessary closed!), parametrized by $t \in [0, 1]$
as well with all points, except $\bar \gamma(0), \bar \gamma(1)$, having a
unique projection to $\gamma(t)$. The point $\bar \gamma(1)$ is called the
image of $\bar \gamma(0)$ from $U(\gamma(0)) \in T_0$. One can show that the
image depends on the homotopy class of the curve rather than on the initial
representative. At the same time, there is a dependence on the choice of the
initial point $\gamma(0)$.

Two vector fields $\xi$ and $\bar \xi$ in a neighborhood of a critical point
are called orbitally equivalent if there exists an analytic coordinate change
$\bar x(x, y), \bar y(x, y)$ such that, in these coordinates, $\xi = f \bar
\xi$ for some analytic function~$f$. If a vector field is linearizable, it is
orbitally equivalent to a linear one.

Now we are ready to prove the proposition. The main theorem of \cite{yoccoz2}
implies that, for $\alpha \in \Sigma_{\mathrm u}$, there exists an analytical
transformation of the complex line in the form
$$
q(z) = \operatorname{exp}(2 \pi \alpha) z + \{\text{higher order terms}\},
$$
which cannot be linearized with analytical coordinate changes. Take this
transformation. The theorem of \cite{yoccoz3} implies that any such
transformation $q(z)$ can be realized as a monodromy mapping for some system of
the form (the system in \cite{yoccoz3} has different sign in front of $u$)
\begin{equation}\label{sys2}
\begin{aligned}
\dot u & = u (1 + \dots), \\
\dot v & = \alpha v (1 + \dots).
\end{aligned}
\end{equation}
Lemma 3.1 in \cite{ie} implies that \eqref{sys2} is orbitally analytically
equivalent to \eqref{sys}.

At the same time, it follows from Theorem 1 in \cite{ie} that  if system
\eqref{sys2} is orbitally equivalent to a linear one, then the monodromy
mappings are conjugate by some transformation. Due to the choice of $q(z)$, we
see that \eqref{sys2} is not orbitally analytically equivalent to a linear
system. Thus, system \eqref{sys} is not equivalent to a linear system and, in
particular, it is not linearizable. The proposition is proved.
\end{proof}

Now we are ready to prove Theorem \ref{t1}. Take an operator field in the form
\eqref{prep2}, where $h(u, v)$ is taken from the statement of Proposition
\ref{p2}. Due to formula \eqref{chng}, the last column of this operator behaves
as a vector field under the coordinate transformation in the triangular form.

The linearization transformation in these coordinates is always triangular;
thus, it exists if and only if the corresponding vector field is linearizable.
Due to the choice of $h(u, v)$, no such analytic coordinate change exists;
thus, the operator field is not linearizable with triangular coordinate change.

Finally, as was mentioned, Lemma \ref{lm1} implies that any $R$ with $\alpha
\in \Sigma_{\mathrm u}$ can be brought to the form \ref{prep2}. Thus, the
constructed operator field $R$ has no linearization coordinates in general form
either. The theorem is proved.

\end{fulltext}

\end{document}